\documentclass[12pt,a4paper]{amsart}
\usepackage{enumerate}
\usepackage{amssymb,latexsym}
\usepackage{amsfonts}
\usepackage{graphicx}
\numberwithin{equation}{section}\theoremstyle{plain}

\newtheorem{theorem}{Theorem}[section]
\newtheorem{proposition}{Proposition}[section]
\newtheorem{lem}{Lemma}[section]

\begin{document}

\title[Discrete harmonic functions in Lipschitz domains]
{Discrete harmonic functions \\
in Lipschitz domains }
\date{}
\author[S. Mustapha]{Sami Mustapha}
\address{Sami Mustapha, Institut Math\'ematiques de Jussieu,
Sorbonne Universit\'e , Tour 25 5e \'{e}tage Boite 247. 4, place
Jussieu F-75252 PARIS CEDEX 05. } \email{sami.mustapha@imj-prg.fr}
\author[M. Sifi]{Mohamed Sifi}
\address{Mohamed Sifi, Universit\'e de Tunis El Manar,
Facult\'e des Sciences de Tunis, LR11ES11 Laboratoire d'Analyse
Math\'ematique et Applications LR11ES11. 2092, Tunis, Tunisie. }
\email{mohamed.sifi@fst.utm.tn}
\renewcommand{\thefootnote}{}



\begin{abstract}
We prove  the existence and uniqueness of a discrete nonnegative
harmonic function for a random walk satisfying finite range,
centering and ellipticity conditions, killed when leaving a globally
Lipschitz domain in $\mathbb{Z}^d$. Our method is based on a systematic use of comparison arguments and discrete potential-theoretical
techniques.
\end{abstract}
\maketitle

\section{Introduction and main results}

Random walks conditioned to live in  domains
$\mathcal{C}\subset\mathbb{Z}^d$ are of growing interest because of
the range of their applications in
enumerative combinatorics, in probability theory  and in harmonic analysis (cf. \cite{B1},
\cite{BM}, \cite{DW}, \cite{FIM}, \cite{FR},
  \cite{R1}). Doob
$h$-transforms, where $h$ is harmonic  for the random walk,
positive within $\mathcal{C}$ and vanishing on its boundary
$\partial\mathcal{C}$, are used to perform such conditioning. It is
therefore crucial to identify the set of all positive harmonic functions
 associated with  a  killed random walk.

General results for homogeneous random walks with non-zero drift killed at the boundary of  a
half-space or  an orthant were obtained in \cite{I1},
\cite{IR}, \cite{KR2}. For random walks with zero drift, only few results are
available \cite{BMS}, \cite{DW}, \cite{GS}, \cite{R1},
\cite{R2}. The first systematical result was obtained by K. Raschel,
who introduced in \cite{R2} a new approach based on the
investigation of a functional equation satisfied by the generating
function of the values taken by the harmonic function. This approach allows him to establish the existence of positive harmonic functions for random walks with small steps and zero drift killed at the boundary of the quadrant $\mathbb{N}^2$. It should be also mentioned that  \cite{R2} provides explicit expressions for these harmonic
functions.

In a recent work Ignatiouk-Robert \cite{I2} investigated the properties of harmonic functions for random walks in via  ladder heights. Applying her general results to random walk in a convex cone she deduced the uniqueness (up to a multiplicative constant) of the harmonic function constructed by Denisov and Wachtel in \cite{DW} under some moment condition on the jumps. Alternative constructions of this harmonic function are proposed by Denisov and Wachtel in \cite{DW1}. These new constructions allow them to remove quite restrictive extendability assumption imposed in \cite{DW}. In \cite{RT} Raschel and Tarrago studied the behavior of the Green function for random walks in convex cone which gives the uniqueness of the harmonic function (see also \cite{DuW}.

Regarding spatially inhomogeneous  random walks the problem is more difficult. \\
 Uniqueness of
positive harmonic functions for random walks with symmetric
spatially inhomogeneous increments, killed at the boundary of a half
space, was established in \cite{M2}  and more recently in the case of an orthant \cite{Bou}.

The main purpose of the present paper is to  extend the results of \cite{Bou}  for the whole class of spatially inhomogeneous centered
random walks satisfying finite span
and ellipticity conditions and killed when leaving a globally Lipschitz unbounded
domain in $\mathbb{Z}^d$.

Consider  $\Gamma\subset \mathbb{Z}^d$  a
finite subset of $\mathbb{Z}^d$ and  let $\pi: \mathbb{Z}^d\times
\Gamma\rightarrow [0,1]$ such that
$$\sum_{e\in\Gamma}\pi(x,e)=1,\quad \sum_{e\in\Gamma}\pi(x,e)e=0;\quad e\in\Gamma,\:
x\in\mathbb{Z}^d.$$ Then, we let $\{S(n),\, n\in \mathbb{N}\}=(S_n)_{n\in\mathbb{N}}$ be the
Markov chain on $\mathbb{Z}^d$ defined by
$$\mathbb{P}[S_{n+1}=x+e/S_n=x]=\pi(x,e);\,\,\quad e\in\Gamma,\:
x\in\mathbb{Z}^d,\: n=0,1,\ldots$$ $(S_n)_{n\in\mathbb{N}}$ is a
centered random walk with bounded increments which becomes spatially
homogeneous if we assume the probabilities $\pi(x,e)$ are
independent of $x$. We shall assume that the set $\Gamma$ contains
 all unit vectors in $\mathbb{Z}^d$, i.e. all the vectors
$e_k=(0, \ldots,0,1,0,\ldots0)\in
\mathbb{Z}^d$, where the $1$ is the $k$-th component.
We shall impose to the random walk
$(S_n)_{n\in\mathbb{N}}$ to satisfy the following uniform ellipticity
condition:
\begin{equation}\label{0} \pi(x,e)\geq\alpha,\quad e\in\Gamma,\:
x\in\mathbb{Z}^d,\end{equation} for some $\alpha>0$.

We shall denote by:

 $\bullet\quad\mathcal{C}$ a globally Lipschitz domain of
$\mathbb{Z}^d$ that is, a domain $\mathcal{C}=\mathcal{D}\cap
\mathbb{Z}^d$ where
$$\mathcal{D}=\left\{(x_1,x')\in\mathbb{R}\times\mathbb{R}^{d-1};x_1>
\varphi(x')\right\}$$ for some Lipschitz function on
$\mathbb{R}^{d-1}$ satisfying
$$|\varphi(x')-\varphi(y')|\leq A|x'-y'|,\quad x',y'\in
\mathbb{R}^{d-1},$$ for some $A>0$, where $|.|$ denote the Euclidean norm. We shall assume that $\varphi(0)=0$.

$\bullet\quad\tau$ the first exit time from $\mathcal{C}$, i.e.,
$$\tau=\inf\{n=0,1,\ldots;\,\,S_n\notin\mathcal{C}\}.$$

$\bullet\quad G_x^y$, $x,y\in\mathcal{C}$,  the Green function
defined by $$G_x^y=\sum_{n\in \mathbb{N}}\mathbb{P}_x(S_n=y,
\tau>n).$$

We are interested in positive functions $h$ which are discrete
harmonic for the random walk $(S_j)_{j\in\mathbb{N}}$ killed at the
boundary of $\mathcal{C}$, i.e. in functions
$h:\overline{\mathcal{C}}\rightarrow \mathbb{R}_+$ such that:
\begin{itemize}\item [\textit{i)}] For any $x\in \mathcal{C}$,
$\,h(x)=\displaystyle \sum_{e\in\Gamma}\pi(x,e)h(x+e)$;
\item [\textit{ii)}] If $x\in\partial\mathcal{C}$,  then $h(x)=0$;
\item [\textit{iii)}] If $x\in\mathcal{C}$,  then $h(x)>0$;
\end{itemize}
where
$\, \overline{\mathcal{C}}=\partial\mathcal{C}\cup\mathcal{C}$.
The boundary of a set $A\subset \mathbb{Z}^d$ is defined by
$$\partial A=\{x\in A^c, x=z+e\;\mbox{for some}\; z\in A\; \mbox{and}\; e\in \Gamma\}$$ and  $\overline{A}=A\cup \partial A$.
In terms of the first exit time of the random walk from
$\mathcal{C}$, we have that
$$h(x)=\mathbb{E}^x\left(h(S_1),\tau>1\right),\quad
x\in\mathcal{C}.$$\begin{theorem}\label{thm1} Let
$(S_n)_{n\in\mathbb{N}}$ be a centered random walk
satisfying the above finite support and ellipticity
conditions. Assume that $\mathcal{C}$ is a globally Lipschitz domain
of $\mathbb{Z}^d$. Then, up to a multiplicative constant, there
exists a unique positive function, harmonic for the random walk
killed at the boundary.
\end{theorem}
The previous result has an important consequence on the Martin
boundary theory attached to the random walk $(S_n)_{n\in\mathbb{N}}$
killed on the boundary of $\mathcal{C}$. Recall that for a transient
Markov chain on a countable state space $E$, the Martin
compactification of $E$ is the unique smallest compactification
$E_M$ of the discrete set $E$ for which the Martin kernels
$y\rightarrow \displaystyle k_y^x=G_y^x/G_y^{x_0}$ (where  $x_0$
is a given reference state in $E$)
extend continuously for all $x\in E$. The minimal Martin boundary
$\partial_mE_M$ is the set of all those $\gamma\in\partial E_M$ for
which the function $x\rightarrow k_\gamma^x$ is minimal harmonic.
Recall that a harmonic function $h$ is minimal if $0\leq g\leq h$
with $g$ harmonic implies $g=ch$ with some $c>0$. By the
Poisson-Martin boundary representation theorem, every nonnegative
harmonic function $h$ can be written as
$$h(x)=\int_{\partial_mE_M}k_\gamma^x\mu(d\gamma),$$ for a some
positive Borel measure $\mu$ on $\partial_mE_M$  (cf. \cite{Dy},  \cite{Ma},  \cite{NY}).
\\

An immediate consequence of Theorem \ref{thm1} is the following.
\begin{theorem} \label{thm2} For all transient random walks satisfying centering, finite
support and ellipticity conditions and all global Lipschitz domains
of $\mathbb{Z}^d$, the minimal Martin boundary is reduced to one
point.
\end{theorem}

We conclude this introduction with some comments which may be helpful  in placing
the results of this paper in their proper perspective.

{\it (i)} The proof of Theorem \ref{thm1} given in \cite{Bou} uses in a crucial way the parabolic Harnack principle. We noted in \cite{Bou} that a more satisfactory approach should dispense with parabolic information and restrict to elliptic tools. A way to get round the difficulties encountered in \cite{Bou} is to use a lower estimate for superharmonic extensions of  discrete positive harmonic functions derived by Kuo and Trudinger in \cite{KT0}. This lower estimate encompasses three powerful ingredients: the Aleksandrov-Bakel'man-Pucci's maximum principle, a barrier technique  and a Calder\'{o}n-Zygmund  covering argument. Going trough the superharmonic extension gives an alternative to the use of \cite[Lemma 2.5]{Bou} and
 provides a purely elliptic derivation of \cite[Proposition 2.6]{Bou} . An advantage of this approach is that it allows us to relax the assumptions $0\in \Gamma$ and $\Gamma=-\Gamma$ made in \cite{Bou}.

{\it (ii)} In case of homogeneous symmetric random walks on unbounded Lipschitz domains, the main results of this paper follows from \cite{GS}. Although  the work of Gyrya and Saloff-Coste concerns diffusion on Dirichlet spaces, to derive the desired results for symmetric random walks, it suffices to consider the corresponding cable process (see \cite[\S 2]{BB}). Since the harmonic functions for cable process and the random walk on the corresponding  graph are essentially the same one has all the desired results (namely Theorem  \ref{thm1}, Theorem  \ref{thm2} and Theorem  \ref{thm6}).

{\it (iii)} Spatially inhomogeneous random walks can
 be considered as the discrete analogues of diffusions generated by second-order differential operators in nondivergence form.
 As in \cite{Bou}, the main tools in this paper are discrete versions of Carleson estimate and boundary Harnack inequality (cf. \cite{BaBu},
\cite{Bau}, \cite{FS}, \cite{FSY}).

{\it (iv)} We restrict ourselves in this paper to  random walks in Lipschitz domains. However, the proofs given below should work for a larger class
of domains, for instance uniform or inner uniform domains (cf. \cite{Ai}).

\section{Proof of Theorem \ref{thm1}}
\subsection{Harnack principle}  We say that a function
$u:\overline{A}=A\cup\partial A\rightarrow \mathbb{R}$ is harmonic in $A\subset\mathbb{Z}^d$
if $Lu=0$ in $A$, where $L$ is the difference operator defined by
$$Lu(x)=\sum_{e\in \Gamma}\pi(x,e)u(x+e)-u(x).$$
In addition to an obvious maximum principle, harmonic functions
satisfy, when they are positive, a Harnack principle. For convenience this principle  is formulated in balls. The discrete Euclidean ball of center $y\in\mathbb{Z}^d$ and radius $R\geq 1$ is denoted $B_R(y)$ and simply $B_R$ when $y$ is clearly understood. We shall also have to use cubes. The cube of center $y\in\mathbb{Z}^d$ and sides $2R$, parallel to the coordinate axes is denoted $Q_R(y)$ and simply $Q_R$ when $y$ is clear.
   The following
theorem (see \cite{KT0} and \cite{KT01}) is a centered version of Harnack principle
established by Lawler  \cite{L}
for random walks with
symmetric bounded increments (as well homogeneous and
inhomogeneous).
\begin{theorem} \label{thm3} {\bf (Harnack principle)}  Assume that $u$ is a nonnegative
harmonic function associated to a random walk satisfying centering,
finite support and uniform ellipticity conditions in a ball $B_{2R}(y)$. Then
$$\max_{B_R(y)}u \leq C\,\displaystyle{\min_{B_R(y)}}\,u,$$ where
$C=C(d,\alpha,\Gamma)>0$.
\end{theorem}
\subsection{Carleson estimate} The classical  Carleson estimate  \cite{C} asserts that a
positive harmonic function vanishing on a portion of the boundary
is bounded, up to a smaller portion, by the value at a fixed point
in the domain with a multiplicative constant independent of the
function.
\begin{theorem} \label{thm5}  Assume that $u$ is a
nonnegative harmonic function in $\mathcal{C}\cap B_{3R}(y)$.
Assume that $u=0$ on $\partial\mathcal{C}\cap B_{2R}(y)$. Then
\begin{equation}\label{1}
    \max\left\{u(x),\,\,x\in \mathcal{C}\cap B_{R}(y)\right\}\leq C\:
    u(y+Re_1),\quad R\geq C,
\end{equation}
where $C=C(d,\alpha,\Gamma,A)>0$ is independent of $y
, R$ and $u$.
\end{theorem}
The proof of Theorem \ref{thm1} relies on the following Proposition.
\begin{proposition} \label{prop1} Let $y\in
\partial\mathcal{C}$ and  $R$ large enough. Let $u$ be a
nonnegative harmonic function in $\mathcal{C}\cap B_{3\sqrt{d}R}(y)$ which
vanishes on $\partial\mathcal{C}\cap B_{2\sqrt{d}R}(y)$. Then
\begin{equation}\label{7}
\max\left\{u(x),\,\,\,x\:\in\:\overline{\mathcal{C}\cap B_{R}(y)}\right\}\leq
\rho \max\left\{u(x),\,\,\,x\:\in\:{\overline{\mathcal{C}\cap
B_{2\sqrt{d}R}(y)}}\right\},
\end{equation}
with a constant $0<\rho=\rho(d,\alpha,\Gamma,A)<1$.
\end{proposition}
\begin{proof}
 To prove \eqref{7} we first observe that it suffices to show that
\begin{equation}\label{71}
    \max\left\{u(x),\: x\in \overline{\mathcal{C}\cap Q_R(y)}\right\}\leq \rho \max \left\{u(x),\; x\in \overline{\mathcal{C}\cap Q_{2R}(y)}\right\}.
\end{equation}
Without loss of generality, we assume $y=0$ and $\max \left\{u(x),\; x\in \overline{\mathcal{C}\cap Q_{2R}}\right\}=1$. Then considering the function $v: \overline{Q_{2R}}\rightarrow\mathbb{R}$ defined by $v=1-u$ in $\overline{\mathcal{C}\cap Q_{2R}}$ and $v=1$ on $\overline{Q_{2R}}\,\backslash\,\overline{\mathcal{C}\cap Q_{2R}}$, we see that \eqref{71} reduces to the following lower estimate
\begin{equation}\label{72}
    v(x)\geq \lambda=\lambda(d,\alpha,\Gamma, A)>0,\quad x\in \overline{Q_R}.
\end{equation}
Since $v$ is superharmonic in $Q_{2R}$ (i.e $Lv\leq 0$ in $Q_{2R}$) we can use the estimate \cite[3.24]{KT01} and deduce that
\begin{equation}\label{73}
    \min_{\overline{Q_R}}v\geq \gamma\left(\frac{\left|\overline{Q_R}\cap \{v\geq 1\}\right|}{|\overline{Q_R}|}\right)^{\frac{\log\gamma}{\log \delta}}
\end{equation}
where $0<\gamma,\delta<1$ are two positive constants depending on $d,\alpha$ and $\Gamma$ and where the notation $|S|$ is used to denote the cardinality of a subset $S\subset\mathbb{Z}^d$.

On the other hand, the fact that $\mathcal{C}$ is Lipschitz allows us to find a circular cone $\mathcal{C}'$ with vertex at the origin such that $\mathcal{C}'\subset \mathcal{C}^c$.
It follows then that there exists a positive constant $\mu$ (depending on $A$) such that for $R$ large enough
\begin{equation}\label{74}
    \left|\overline{Q_R}\cap \{v\geq 1\}\right|\geq\left| \overline{Q_R}\cap \mathcal{C}'\right|\geq \mu |\overline{Q_R}|.
\end{equation}
We conclude from \eqref{73} and \eqref{74} that
$$\min_{\overline{Q_R}}v\geq \gamma^{1+\frac{\log\mu}{\log\delta}},$$
which implies \eqref{72} and completes the proof of \eqref{71}.\end{proof}
 \noindent{\it Proof
of Theorem \ref{thm5}. } To prove the Carleson estimate \eqref{1} we
first observe that the uniform ellipticity assumption implies that $u(\xi)
\leq Ce^{C\,|\xi-\zeta|}u(\zeta)$, $\xi, \zeta \in \mathcal{C}\cap
B_{3R}(y)$; where $C=C(d,\alpha,\Gamma,A)>0$. This local Harnack
principle allows us to assume that the distance of $x$ from
$\partial\mathcal{C}$ is sufficiently large. We shall denote by
$\delta(x)$ ($x\:\in\:\mathcal{C}\cap B_{2R}(y)$) this distance and
suppose that $\delta(x)\geq C$. The fact that $\mathcal{C}$ is
Lipschitz combined with Harnack principle (Theorem \ref{thm3}) imply
that
\begin{equation}\label{8}
    u(x)\leq C\, \left(\frac{R}{\delta(x)}\right)^\gamma
    u(y+Re_1),\quad x\:\in\:\mathcal{C}\cap B_{2R}(y),
\end{equation}
where $\gamma$ and $C$ are positive constants depending on
$d,\alpha,\Gamma$ and $A$.

  Let $x\in \mathcal{C}\cap {B_{2R}(y)}$,
and let us assume that
\begin{equation}\label{9}
    \delta(x) < \left(1-\left(\frac{1+\rho}{2}
    \right)^{\frac{1}{\gamma}}\right)
    \frac{2R-|x-y|}{8\sqrt{d}}
\end{equation}
where $\rho$ is the constant obtained in (\ref{7}) and $\gamma$ the
exponent that appears in (\ref{8}).
Let $x_0\in \partial\mathcal{C}$  such
that $|x-x_0|=\delta(x)$. It follows easily from (\ref{9}) and the
fact that $\delta(x)$ is sufficiently large that
$\overline{B_{3\sqrt{d}\delta(x)}(x_0)}\subset B_{2R}(y)$. By Proposition \ref{prop1}
applied to the harmonic function $u$ in the domain
$B_{{3}\sqrt{d}\delta(x)}(x_0)\cap\mathcal{C}$, we have
\begin{equation}\label{10}
u(x)\leq \max\left\{u(x'),\,\,\, x'\in
\mathcal{C}\cap\overline{B_{\frac{3}{2}\delta(x)}(x_0) }\right\}
\leq\rho \max\left\{u(x'),\,\,\, x'\in \mathcal{C}\cap
{\overline{B_{3\sqrt{d}\delta(x)}(x_0)}} \right\}
.\end{equation} Let $z\in \mathcal{C}\cap
\overline{B_{3\sqrt{d}\delta(x)}(x_0)}$ satisfying
$$u(z)=\max\left\{u(x'),\,\,\, x'\in \mathcal{C}\cap\overline{B_{3\sqrt{d}\delta(x)}(x_0)
} \right\}.$$
 We have
$$(2R-|x-y|)\leq (2R-|z-y|)+8\sqrt{d}\delta(x). $$
Hence, thanks to \eqref{9}
$$(2R-|x-y|)\leq\left(\frac{1+\rho}{2}\right)^{-\frac{1}{\gamma}}(2R-|z-y|).$$
It follows that
$$(2R-|x-y|)^\gamma u(x)\leq\left(\frac{1+\rho}{2}\right)^{-1}(2R-|z-y|)^\gamma
u(x)
,$$\\
and therefore, by \eqref{10}
\begin{eqnarray}\label{11}
(2R-|x-y|)^\gamma u(x)&\leq& \frac{2\rho}{1+\rho} (2R-|z-y|)^\gamma
u(z)\\ &\leq& \nonumber  \theta_0\max_{x'\,\in {\mathcal{C}\cap
B_{2R}(y)}}(2R-|x'-y|)^\gamma u(x')\end{eqnarray} where
$$\theta_0=\displaystyle\frac{2\rho}{1+\rho}<1.$$

It remains to consider the case where

\begin{equation}\label{12}
    \delta(x)\geq \left(1-\left(\frac{1+\rho}{2}\right)^{\frac{1}{\gamma}}\right)
    \frac{2R-|x-y|}{8\sqrt{d}}.
\end{equation}
  In follows from (\ref{12}) that
$$(2R-|x-y|)^\gamma u(x)\leq\varepsilon_0^{-\gamma}\delta(x)^\gamma u(x)
\leq \varepsilon_0^{-\gamma}\max_{x'\,\in\mathcal{C}\cap
B_{2R}(y)}\delta(x')^\gamma u(x')$$ where
$8\sqrt d \varepsilon_0=\displaystyle
1-\left(\frac{1+\rho}{2}\right)^{\frac{1}{\gamma}}$ and, thanks to
\eqref{8},
\begin{equation}\label{13}
   (2R-|x-y|)^\gamma u(x)\leq C\,\varepsilon_0^{-\gamma}R^\gamma
    u(y+Re_1).
\end{equation}

Putting together \eqref{11} and \eqref{13} and taking the supremum
over $\mathcal{C}\cap B_{2R}(y)$, we deduce that
$$
  \max_{x\,\in {\mathcal{C}\cap
B_{2R}(y)}}(2R-|x-y|)^\gamma u(x) \leq  \theta_0 \max_{x\,\in
{\mathcal{C}\cap B_{2R}(y)}}(2R-|x-y|)^\gamma u(x)
+C\,\varepsilon_0^{-\gamma}R^\gamma
    u(y+Re_1) .$$
Using the fact that  $(2R-|x-y|)\approx R$ for
$x\,\in\mathcal{C}\cap B_{R}(y)$ we deduce the estimate
\eqref{1}.$\hfill\Box$ \bigskip
\subsection{Boundary Harnack principle}

Carleson estimate can be extended to the ratio $u/v$ of
positive harmonic functions.

\begin{theorem}\label{thm6} {\bf (Boundary Harnack principle)} Let
$y\in\partial\mathcal{C}$ and $K>0$ large enough. Assume that $u$
and $v$ are two nonnegative harmonic functions in $\mathcal{C}\cap
B_{3KR}(y)$.
 Assume that $u,\, v =0$ on
$\partial\mathcal{C}\cap B_{2KR}(y)$. Then
\begin{equation}\label{14}
    \max_{x\,\in\:\mathcal{C}\cap B_R(y)}\frac{u(x)}{v(x)}\leq C\:
    \frac{u(y+Re_1)}{v(y+Re_1)},\quad R\geq C,
\end{equation}
where $C=C(d,\alpha,\Gamma,A,K)>0$.
\end{theorem}
The above formulation of the boundary Harnack principle follows the classical formulation but
the proof of  \eqref{14} which will be given below shows
that the assumption $v=0$ on $\partial\mathcal{C}\cap B_{2R}(y)$ is
not needed so that \eqref{1} constitutes a special case of
\eqref{14}.

The estimate \eqref{14} is an immediate consequence of the lower estimate contained in the following lemma.

For $y\in \partial \mathcal{C}$ and $R\geq r\geq 1$, we
shall denote by
\begin{eqnarray*}
  \mathcal{D}_{R,r}(y)&=&B_R(y)\cap \{x\in \mathcal{C},\delta(x)>r\}; \\
  \mathcal{C}_{R,r}(y)&=&\left(B_R(y)\cap \mathcal{C}\right) \,\setminus\,\mathcal{D}_{R,r}(y).
\end{eqnarray*}
For $R\geq r\geq 1$, the boundary of $\mathcal{C}_{R,r}$ is the union of three sets: the ``bottom'' $\partial\mathcal{C}_{R,r}\cap \mathcal{C}^c$, the ``lateral side'' $\partial\mathcal{C}_{R,r}\cap\{x\in\mathcal{C},\:0\leq \delta(x)\leq r\}$ and the ``top'' $\partial\mathcal{C}_{R,r}\cap \mathcal{D}_{R,r}$.
\begin{lem}\label{lem2} There exists a constant $K_0>0$ such that for all $K\geq K_0$ and for all $y\in \partial\mathcal{C}, $ $r\geq 1$,
\begin{equation}\label{15}
\min_{x\in \mathcal{C}\cap B_r(y)}\frac{\mathbb{P}_x
\left[S(\tau_{\mathcal{C}_{Kr,r}(y)})\in \mathcal{D}_{Kr,r}(y)
\right]}{\mathbb{P}_x \left[S(\tau_{\mathcal{C}_{Kr,r}(y)})\in
\left(\partial\mathcal{C}_{Kr,r}(y)\cap\mathcal{C}\right)\,\setminus\,
\mathcal{D}_{Kr,r}(y) \right]}\geq 1
\end{equation} where $\tau_{\mathcal{C}_{Kr,r}(y)}$ denotes
the exit time from $\mathcal{C}_{Kr,r}(y)$.
\end{lem}
\noindent{\it
 Proof of Theorem \ref{thm6}.}
In order to derive estimate \eqref{14} from \eqref{15}, we first
observe that it is always possible to assume that
$u(y+Re_1)=v(y+Re_1)=1$. For a large $R$, Carleson estimate
\eqref{1} implies that the function
$u$ is dominated by a positive constant $c_0$ in the region
$B_{KR}(y)\cap \mathcal{C}$. This constant $c_0$ can be chosen so that by Harnack principle the lower estimate
$v\geq \frac{1}{c_0}$ holds on $\mathcal{D}_{KR,R}(y)$. Let
$v_0=c_0v$ and $u_0=\displaystyle \frac{u}{c_0}-v_0$. Let $x\in
B_{R}(y)\cap \mathcal{C}$. We have:
\begin{eqnarray*} u_0(x) &\leq &
\mathbb{P}_x\left[S(\tau_{\mathcal{C}_{KR,R}(y)})\in
\left(\partial\mathcal{C}_{KR,R}(y)\cap\mathcal{C}\right)\,\setminus\, \mathcal{D}_{KR,R}(y)\right]\\
&\leq & \mathbb{P}_x\left[S(\tau_{\mathcal{C}_{KR,R}(y)})\in
\mathcal{D}_{KR,R}(y)\right]\\
& \leq & v_0(x),
\end{eqnarray*} where the second inequality follows from \eqref{15}. We deduce then that
$$\frac{u(x)}{c_0}-c_0v(x)\leq c_0 v(x),
\quad x\in B_{R}(y)\cap \mathcal{C}.$$ So that
$$\frac{u(x)}{v(x)}\leq 2 c_0^2,\quad x\in B_{R}(y)\cap \mathcal{C},$$
which completes the proof of \eqref{14}. $\hfill\Box$
\\
{\it Proof of Lemma \ref{lem2}.} To prove estimate
\eqref{15} it suffices to show that if  $u,\, v :
\overline{\mathcal{C}_{Kr,r}(y)}\rightarrow \mathbb{R}$ (where $y\in
\partial\mathcal{C}$, $r\geq 1$ are fixed) satisfy
\begin{equation}\label{16}
\left\{
\begin{array}{cc}
 \displaystyle u(x)=\sum_{e\in \Gamma}\pi(x,e)u(x+e),  & x\in\mathcal{C}_{Kr,r}(y) \\
  u(x)\geq 0  & \mbox{in} \quad {\overline{\mathcal{C}_{Kr,r}(y)}} \\
  u(x)\geq 1 &\mbox{ on} \quad \partial\mathcal{C}_{Kr,r}(y)\cap \mathcal{D}_{Kr,r}(y)
\end{array}\right.
\end{equation}
\begin{equation}\label{17}
\left\{\begin{array}{ccc} &\displaystyle v(x)=
\sum_{e\in \Gamma}\pi(x,e)v(x+e),& \quad x\in\mathcal{C}_{Kr,r}(y)\\
&v(x)\leq 1 & \quad \mbox{in} \quad
{\overline{\mathcal{C}_{Kr,r}(y)}}\\
& v(x)\leq 0 & \quad \mbox{ on} \quad
\partial\mathcal{C}_{Kr,r}(y)\cap \mathcal{D}_{Kr,r}(y)
\end{array}\right.
\end{equation}
then we have
\begin{equation}\label{18}
    v(x)\leq u(x), \quad x\in B_r(y)\cap\mathcal{C}
\end{equation}
provided that $K\geq K_0$ is large enough.

First we prove that under \eqref{16} the function $u$ satisfies
\begin{equation}\label{19}
u(x) \geq 2 \alpha \left( \frac{ \delta(x)}{r}\right)^\beta,
\,\,\,\,\, x \in B_r(y)\cap \mathcal{C},
\end{equation}
for appropriate constants $\alpha, \, \beta >0.$

The proof of  \eqref{19} relies on the following construction.

We assume $K$ large enough and we define $ \tilde u :  {\overline {
B_{Mr}(y) \cap \mathcal{C}}} \longrightarrow \mathbb{R} $ by
\begin{equation*}
\left\{
\begin{array}{cc}
 \displaystyle \tilde u(x)=\sum_{e\in \Gamma}\pi(x,e)
 \tilde u(x+e),  & x\in  B_{Mr}(y) \cap \mathcal{C} \\
  \tilde u  = \min(u,1)  & \mbox{on} \quad
\partial \left( B_{Mr}(y) \cap \mathcal{C} \right) \,\backslash\,
 { \mathcal{D}_{Kr,r}(y)} \\
  \tilde  u = 1  &\mbox{on} \quad
\partial \left( B_{Mr}(y) \cap \mathcal{C} \right)
 \cap
 { \mathcal{D}_{Kr,r}(y)}
\end{array}\right.
\end{equation*}
where $0<M \leq K $ is chosen so that $\tilde y = y +(M-1) re_1$
satisfies
\begin{equation*}\label{20}
 \delta( \tilde y) \geq 10 r.
\end{equation*}

Let $\mathcal{U} = \left( B_{Mr}(y) \cap \mathcal{C} \right) \cap
B_{2r}(\tilde{y})$ and $w: \overline{B_{2r}(\tilde{y})}
\longrightarrow \mathbb{R} $ be defined by
$$ w(x) =  \tilde u (x), \,\,\,\, x\in \overline{\mathcal{U}};\; w(x)=1,\quad x \in \mathcal{U}^c\cap \overline{B_{2r}(\tilde{y})}. $$
It is easy to see that $w$ is superharmonic. Let $\tilde{z}=y+(M-2)re_1$. By the same argument used
in the proof of the lower estimate \eqref{72} combined with Harnack principle, we see that $w(\tilde{z})$ satisfies a lower estimate $w(\tilde{z})\geq c$. It follows then that $ \tilde{u}(\tilde{z})\geq c$.

Since $u\geq 1$ on
$\partial\mathcal{C}_{Kr,r}(y)\cap \mathcal{D}_{Kr,r}(y)$, we deduce
by the maximum principle that $u\geq \tilde{u}$ on
$B_r(y)\cap\mathcal{C}$. Combining with Harnack inequality we de
deduce \eqref{19}.

It follows from \eqref{19} that if $x\in
\overline{\mathcal{C}_{r,r}(y)}\,\backslash\,
\mathcal{C}_{r,r/K}(y)$ then we have
\begin{equation}\label{22}
    u(x)\geq 2\alpha K^{-\beta}.
\end{equation}

 Let us now prove that there exists $N>0$ such that
\begin{equation}\label{23}
    v(x)\leq e^{-NK}, \quad x \in \overline{\mathcal{C}_{r,r}(y)}.
\end{equation}
Let $j=1,\ldots, \lfloor\frac{K-1}{2}\rfloor$  and let $x_j\in
\partial\mathcal{C}_{(2j-1)r,r}(y)$ be such that
$$v(x_j)=\max\left\{v(x),\quad x \in
\overline{\mathcal{C}_{(2j-1)r,r}(y)}\right\}.$$ Let $\mathcal{U}_j=
B_{2r}(x_j)\cap {\mathcal{C}_{(2j+1)r,r}(y)}$ and
$\tau_{\mathcal{U}_j}$ be the exit time from $\mathcal{U}_j$. By the
same argument used in the proof of \eqref{72} we see that
$$\mathbb{P}_{x_j}\left[S(\tau_{\mathcal{U}_j})\in \mathcal{D}_{Kr,r}(y)
\right]\geq c>0.$$ Using \eqref{17} (in particular, the fact that
$v\leq0$ on
$\partial\mathcal{C}_{Kr,r}(y)\cap \mathcal{D}_{Kr,r}(y)$) we deduce
then that
$$v(x_j)\leq\theta\max_{
\overline{\mathcal{U}_j}}v,$$ where $0<\theta<1$. Hence
$$\max\left\{ v(x),\quad x\in
\overline{\mathcal{C}_{(2j-1)r,r}(y)}\right\}\leq\theta \max\left\{
v(x),\quad x\in \overline{\mathcal{C}_{(2j+1)r,r}(y)}\right\}.$$
Iterating this estimate we obtain
$$\max\left\{ v(x),\quad x\in
\overline{\mathcal{C}_{r,r}(y)}\right\}\leq\theta^{\lfloor\frac{K-1}{2}\rfloor} \max\left\{ v(x),\quad
x\in \overline{\mathcal{C}_{Kr,r}(y)}\right\}\leq e^{-NK},$$
 which proves
\eqref{23}. It follows from \eqref{23} that
\begin{equation}\label{231}
v\leq\alpha
K^{-\beta}\quad \mbox{in}\quad \overline{\mathcal{C}_{r,r}(y)}\end{equation}
provided that $K$ is large enough.

From the previous considerations it follows that
$$u_1=\frac{K^\beta}{2\alpha}u\geq 0 \quad \mbox{ in}\quad
\overline{\mathcal{C}_{r,r/K}(y)}$$ with
$$u_1\geq1 \quad
\mbox{in}\quad \partial
\mathcal{C}_{r,r/K}(y)\cap \mathcal{D}_{r,r/K}(y)$$ thanks to
\eqref{22} and, thanks to \eqref{231},
$$v_1=\frac{K^\beta}{2\alpha}(2v-u)\leq \frac{K^\beta}{\alpha}v\leq 1
\quad \mbox{in}\quad\overline{\mathcal{C}_{r,r/K}(y)},$$ with
$$v_1\leq 0  \quad \mbox{ on}\quad \partial\mathcal{C}_{r,r/K}(y)\cap\mathcal{D}_{r,r/K}(y).$$
In particular, we have
$$u_1-v_1= \frac{K^\beta}{\alpha}(u-v) \geq 0 \quad \mbox{ on}\quad
\overline{\mathcal{C}_{r,r}(y)\,\backslash\,\mathcal{C}_{r,r/K}(y)}.$$ It follows  that
$u_1$, $v_1$ satisfy the same assumptions as $u$, $v$ with $r$
replaced by $r/K$. We can then iterate and define $u_i$, $v_i$ such
that
$$u_i-v_i=\left(\frac{K^\beta}{\alpha}\right)^i(u-v) \geq 0
\quad \mbox{ on}\quad \overline{\mathcal{C}_{r/K^i,r/K^i}(y)\,\backslash\, \mathcal{C}_{r/K^i,r/K^{i+1}}(y)}$$
$i=1,2\ldots$. We deduce then that
\begin{equation*}\label{24}
    u-v\geq 0 \quad \mbox{ on}\quad
S(y)=\bigcup_{i\geq
0}\overline{\mathcal{C}_{r/K^i,r/K^i}(y)\,\backslash\,
\mathcal{C}_{r/K^i,r/K^{i+1}}(y)}.
\end{equation*}
 Let $x\in
B_r(y)$ and $\tilde{x}\in\partial\mathcal{C}$ satisfying
$\delta(x)=|x-\tilde{x}|$. Then
$$\mathcal{C}_{Kr,r}(\tilde{x})\subset
\mathcal{C}_{(K+2)r,r}(y).$$ Replacing $K$ by $K+2$ in the
previous considerations we deduce that $u\geq v$ on $S(\tilde{x})$
that contains $x$. This shows that $u(x)\geq v(x)$ and completes the
proof of \eqref{18}. $\hfill \Box$\\ \\
{\it Proof of Theorem \ref{thm1}.} The proof is the same as \cite{Bou}. We observe  that instead of Carleson estimate, we can simply use the estimate
$$u(\xi)\leq C\, e^{C|\xi-\zeta|}u(\zeta),\quad \xi,\zeta\in \mathcal{C},$$
which follows from uniform ellipticity. The advantage of this estimate is that it works for all connected infinite domains, and not just for domains satisfying Carleson estimate. It is already enough to run the diagonal process argument used in \cite{Bou}.

As in \cite{Bou} the uniqueness can be deduced by essentially the same method as in \cite[Proof of Theorem 3]{Ai} and \cite[Lemma 6.2]{An}.
  $\hfill\Box$


\begin{thebibliography}{00}
    \bibitem{Ai} H. \textit{Aikawa, Boundary Harnack principle and Martin boundary
for a uniform domain,} J. Math. Soc. Japan, \textbf{53}
(2001), 119-145.
     \bibitem{An} A. Ancona, \textit{Principe de Harnack \`{a} la fronti\`{e}re et th\'{e}or\`{e}me de Fatou pour un op\'{e}rateur elliptique dans
      un domaine lipschitzien},  Ann. Inst. Fourier (Grenoble) 28 (1978)  169-213.
         \bibitem{BB} M. T. Barlow and R. F. Bass, \textit{Stability of parabolic Harnack inequalities,} Trans. Amer. Math. Soc. \textbf{356} (2004)  1501-1533.
         \bibitem{BaBu}  R. F. Bass and K. Burdzy, \textit{The boundary Harnack principle for nondivergence form elliptic operators,} J. London Math. Soc. \textbf{50} (1994) 157-169.
         \bibitem{Bau} P. E. Bauman, \textit{Positive solutions of elliptic equations in nondivergence form and their adjoints,} Ark. Math. \textbf{22} (1984) 536-565.
          \bibitem{BMS} N. Ben Salem, S. Mustapha and M. Sifi,
    \textit{Potential Theoretic Tools and Randon Walks}, ESAIM: Proccedings and Surveys, to appear.
    \bibitem{B1}  P. Biane, \textit{Quantum random walk on the dual of
    $SU(n)$},  Probab. Theory and Related Fields \textbf{89} (1991)
117-129.
     \bibitem{Bou} A. Bouaziz, S. Mustapha and M. Sifi, \textit{Discrete harmonic functions on an orthant in $\mathbb{Z}^{d}$},
     Electron. Commun. Probab. \textbf{20} (2015)
1-13.
\bibitem{BM} M. Bousquet-M\'elou
and M. Mishna, \textit{Walks with small steps in the quarter plane},
Contemp. Math. \textbf{520} (2010)  1-40.
\bibitem{C} L. Carleson, \textit{On the existence
of boundary values for harmonic functions in several variables,}
Ark. Mat. \textbf{4} (1962) 393-399.
\bibitem{DW} D. Denisov and V. Wachtel,
\textit{Random Walks in Cones}, Ann. Probab.   \textbf{43} (2015)  992-1044.
\bibitem{DW1} D. Denisov and V. Wachtel, \textit{Alternative constructions of a harmonic function for a random walk in a cone}, Preprint 2018, arXiv:1805.01437.
\bibitem{Dy} E. B. Dynkin, \textit{The boundary theory of Markov
processes (discrete case)}, Uspehi Mat. Nauk \textbf{24}  (1969) 3-42.
\bibitem{DuW} J. Duraj and V. Wachtel,  \textit{Green function of a random walk in a cone}, Preprint 2018, arXiv:1807.07360.
\bibitem{FS} E. B. Fabes and M. V. Safonov, \textit{Behavior near the boundary of positive solutions of second order parabolic equations,} J. Fourier Anal. Appl. \textbf{3} (1997) 871-882.
\bibitem{FSY} E. B. Fabes, M. V. Safonov and Y. Yuan, \textit{Behavior near the boundary of positive solutions of second order parabolic equations, II}, Trans. Amer. Math. Soc. \textbf{351} (1999) 4947-4961.
\bibitem{FIM} G. Fayolle,
R. Iasnogorodski and V. Malyshev, \textit{Random walks in the
quarter-plane},
Springer-Verlag, Berlin, 1999.
\bibitem{FR} G. Fayolle and K. Rashel, \textit{ Random walks in the quarter plane with zero drift: an explicit criterion for the finiteness of the associated group},
Markov Processes and Related Fields  \textbf{17}  (2011)  619-636.
\bibitem{GS} P. Gyrya and L. Saloff-Coste, \textit{Neumann and Dirichelet Heat Kernels in Inner Uniform Domains,} Volume \textbf{336}, Ast\'{e}risque, 2011.
   \bibitem{I1}  I. Ignatiouk-Robert, \textit{Martin boundary of a killed random walk on a
half-space,} J. Theoret. Probab. \textbf{21} (2008)  35-68.
\bibitem{I2} I. Ignatiouk-Robert, \textit{Harmonic functions of random walks in a semigroup via ladder heights}, Preprint 2018, arXiv:1803.05682.
\bibitem{IR} I. Ignatiouk-Robert and C. Lor\'{e}e, \textit{Martin boundary of a
killed random walk on a quadrant}, Ann. Probab. \textbf{38} (2010)
1106-1142.
\bibitem{KT0}
H. J. Kuo and
N. S. Trudinger, \textit{Linear differential elliptic difference inequalities with random coefficients,} Math. Comp. 55 (1990), no. 191, 37--53.
\bibitem{KT01}  Kuo, Hung-Ju; Trudinger, Neil S, \textit{Positive difference operators on general meshes,} Duke Math. J. 83 (1996), no. 2, 415-433.
\bibitem{KR2} I. Kurkova and K. Raschel, \textit{Random walks in $\mathbb{Z}_+^2$ with non-zero drift absorbed at the axes,}
Bulletin de la Soci\'et\'e Math\'ematique de France 139 (2011), 341-387.
\bibitem{L} G. F. Lawler, \textit{Estimates for differences and Harnack inequality
for difference operators coming from random walks with symmetric,
spatially inhomogeneous, increments}, Proc. London Math. Soc.  \textbf{63}
(1991)  552-568.
 \bibitem{Ma} R. S. Martin, \textit{Minimal positive harmonic
functions}, Trans. Amer. Math. Soc. \textbf{49} (1941)  137-172.
     \bibitem{M2} S. Mustapha, \textit{Gambler's
ruin estimates for random walks with symmetric spatially
inhomogeneous increments}, Bernoulli  \textbf{13} (2007)  131-147.
\bibitem{NY} Ney, P. and Spitzer, F. (1966). \textit{The Martin boundary for random walk}, Trans.
Amer. Math. Soc. \textbf{121}  116-132.
\bibitem{R1} K. Raschel, \textit{Green functions for
killed random walks in the Weyl chamber of ${\rm Sp}(4)$}, Ann.
Inst. Henri Poincar\'e Probab. Stat. \textbf{47} (2011)   1001-1019.
\bibitem{R2} K. Raschel, \textit{Random walks in the quarter plane,
discrete harmonic functions and conformal mappings, with an appendix
by Sandro Franceschi}, Stochastic Processes and their Applications
\textbf{124} (2014)  3147-3178.
\bibitem{RT} K. Raschel and P. Tarrago, \textit{Martin boundary of random walks in convex cones}, Preprint 2018, arXiv:1803.09253.
     \end{thebibliography}
\end{document}